\documentclass[12pt]{amsart}
\usepackage{amsmath,amsthm,amssymb}
\usepackage{graphicx}
\usepackage{enumerate}
\usepackage[retainorgcmds]{IEEEtrantools}
\usepackage{youngtab}
\usepackage{color}
\usepackage[leqno]{amsmath}
\numberwithin{equation}{section}
\newcounter{keepeqno}

\usepackage[pagebackref=true]{hyperref}
\usepackage[margin=1.5in]{geometry}
\usepackage{xy}
\usepackage{tikz}
\newtheorem{theorem}{Theorem}[section]
\newtheorem{lemma}[theorem]{Lemma}

\newtheorem{proposition}[theorem]{Proposition}

\theoremstyle{definition}
\newtheorem{definition}[theorem]{Definition}

\newtheorem{remark}[theorem]{Remark}
\newtheorem{example}[theorem]{Example}

\DeclareMathOperator{\In}{{\rm in}}
\DeclareMathOperator{\lex}{\rm lex}
\DeclareMathOperator{\grevlex}{\rm grevlex}
\begin{document}
\bibliographystyle{plain}
\newcommand{\br} {{r}}                              
\newcommand{\A} {\mathbb{A}}                           
\newcommand{\PGL} {\Pj\Gl_2(\R)}                           
\newcommand{\RP} {\R\Pj^1}                                 
\newcommand{\C} {{\mathbb C}}                              
\newcommand{\bF}{\mathbb F}
\newcommand{\bR}{\mathbb R}
\newcommand{\cE} {{\mathcal E}}
\newcommand{\cF} {{\mathcal F}}
\newcommand{\cG} {{\mathcal G}}

\newcommand{\cI}{\mathcal I}
\newcommand{\cL} {{\mathcal L}}
\newcommand{\gr} {\mathrm{gr}}
\newcommand{\lev}{\mathrm{lev}}
\newcommand{\inn}{\mathrm{in}}
\newcommand{\R} {{\mathbb R}}                              
\newcommand{\Q}{{\mathbb Q}}                              
\newcommand{\Z} {{\mathbb Z}}                              
\newcommand{\N}{{\mathbb N}}                               
\renewcommand{\k}{{\mathbf k}}                               

\newcommand{\Pj} {{\mathbb P}}                             
\newcommand{\T} {{\mathbb T}}                              
\newcommand{\Sg} {{\mathbb S}}                             
\newcommand{\Gl} {{\rm Gl}}                                
\newcommand{\Sl} {{\rm Sl}}                                
\newcommand{\SL}{{\mathrm SL}}
\newcommand{\fn}{\mathfrak n}

\newcommand{\F}  [2] {\ensuremath{C_{#2}({#1})}}                            
\newcommand{\oM} [1] {\ensuremath{{\mathcal M}_{0,#1}(\R)}}                 
\newcommand{\M} [1] {\ensuremath{{\overline{\mathcal M}}{_{0, #1}(\R)}}}    
\newcommand{\cM} [1] {\ensuremath{{\mathcal M}_{0, #1}}}                    
\newcommand{\CM} [1] {\ensuremath{{\overline{\mathcal M}}_{0, #1}}}         
\newcommand{\bd} [1] {\mathbf{#1}}

\newcommand{\roverm}[1]{\frac{r(#1)}{m}}

\newcommand{\suchthat} {\ \ | \ \ }
\newcommand{\ore} {\ \ {\it or} \ \ }
\newcommand{\oand} {\ \ {\it and} \ \ }


\newcommand{\vw}{\mathbf{w}}
\newcommand{\PPP}[2]{(\Pj^{#1})^{#2}}
\newcommand{\floor}[1]{\lfloor #1 \rfloor}
\newcommand{\ceiling}[1]{\lceil #1 \rceil}
\def\Proj{\mathrm{Proj} \,}
\def\q{/\!/}
\newcommand{\tableau}[2]{\begin{array}{|c|}\hline #1 \\ \hline #2 \\ \hline \end{array}}
\newcommand{\vd}{\mathbf{d}}
\newcommand{\va}{\mathbf{a}}

\def\O{\mathcal{O}}                               
\def\X{\mathcal{X}}
\def\L{\mathcal{L}}
\def\<{\langle}
\def\>{\rangle}
\def\({\left (}
\def\){\right )}
 

\theoremstyle{plain}
\newtheorem{thm}{Theorem}
\newtheorem{prop}[thm]{Proposition}
\newtheorem{cor}[thm]{Corollary}
\newtheorem{lem}[thm]{Lemma}
\newtheorem{conj}[thm]{Conjecture}

\theoremstyle{definition}
\newtheorem{exmp}[thm]{Example}

\theoremstyle{remark}
\newtheorem{defn}[thm]{Definition}
\newtheorem{rem}[thm]{Remark}
\newtheorem*{hnote}{Historical Note}
\newtheorem*{nota}{Notation}
\newtheorem*{ack}{Acknowledgments}
\numberwithin{equation}{section}


\title { 
The ring of evenly weighted points on the line}
\date\today

\keywords{Invariant theory, Gelfand-Tsetlin polytopes, SAGBI degeneration}

\author{Milena Hering}
\address{Milena Hering, Maxwell Institute and School of Mathematics, University of Edinburgh, UK}
\email{m.hering@ed.ac.uk}

\author{Benjamin J.\ Howard}
\address{Benjamin J. Howard,
Center for Communications Research, Institute for Defense Analysis, Princeton, NJ 08540 USA}
\email{bjhowa3@idaccr.org}

\begin{abstract}
Let $M_w = (\Pj^1)^n \q \mathrm{SL}_2$ denote the geometric invariant theory 
quotient of $(\Pj^1)^n$ by the diagonal action of $\mathrm{SL}_2$
using the line bundle $\mathcal{O}(w_1,w_2,\ldots,w_n)$ on 
$(\Pj^1)^n$. 
Let $R_w$ be the coordinate ring of $M_w$. 
We give a closed formula for the  
Hilbert function of $R_w$, which allows us to compute the 
degree of $M_w$. The graded parts of $R_w$ are certain Kostka 
numbers, so this Hilbert function computes stretched Kostka 
numbers. 
If  all the weights $w_i$ are even, 
we find a presentation of $R_w$ so that the ideal $I_w$ of this presentation
has a quadratic Gr\"obner basis. In particular, $R_w$ is Koszul. 
We obtain this result by studying the homogeneous coordinate ring of a 
projective toric variety arising as a degeneration of $M_w$.
\end{abstract}

\thanks{
 The first author
   was partially supported by an Oberwolfach Leibniz Fellowship and 
   NSF grant
   DMS 1001859.}
   
\maketitle

\section{Introduction}

The study of the ring of invariants for the action of the automorphism group 
of $\Pj^1$ on $n$ points on $\Pj^1$ goes back to the 19th century.
In 1894 
Kempe \cite{Kempe}  proved that this ring  
is generated  by the invariants 
of lowest degree. 
More than a century later Howard, Millson, Snowden, and 
Vakil \cite{HMSV12} 
were finally able to describe the ideal of relations between Kempe's generators, when the characteristic of the ground field $\k$ is zero or $p > 11$. 

More generally, for 
$w=(w_1, \ldots, w_n) \in \Z^n$, let $L_w= \O_{(\Pj^1)^n}(w_1, 
\ldots, w_n)$. Assume that all $w_i$ are positive, so that 
$L_w$ is very ample. The group $\SL(2)$ acts diagonally on 
$(\Pj^1)^n$ and the line bundle $L_w$ admits a unique linearization. 
Let 
\[ R_w = \left( \bigoplus _{d\geq 0} H^0\left( (\Pj^1)^n, L_w^d\right) 
\right) ^{\SL(2)}\]
denote the corresponding ring of invariant sections, and let 
$M_w = (\Pj^1)^n\q\SL(2)$ denote the GIT quotient. When $w_i=1$ for $1\leq i\leq n$, we write $w=1^n$. 

In 
\cite[Theorem 2.3]{HMSV}
the authors show that $R_w$ is generated by the invariants of lowest degree for
arbitrary $w$ 
 and in \cite[Theorem 1.1]{HMSV12} that, in characteristic zero or $p > 11$, 
the ideal of relations $I_w$ is generated by quadratic polynomials in the generators 
unless $w=1^6$, in which case there is an essential cubic relation.
Moreover, in \cite[Section 2.15]{HMSV}, 
the authors obtain a recursive formula for the degree 
of $M_w$.

Our first theorem is an extension of Howe's formula  \cite[5.4.2.3]{Howe88} 
for the Hilbert function of $R_w$  in the case $w=1^n$ to arbitrary 
$w$. In particular, we obtain a closed formula for the degree of $M_w$.  

\begin{theorem}\label{thm:hilbertfn}
Let $[n] = \{ 1, \ldots, n\}$, 
and for $J\subseteq [n]$, 
set
 $|w_J| = \sum_{j\in J}w_j$,  
$w_{\emptyset}=0$ and $|w|=w_1+\cdots +w_n$.
\begin{enumerate}
\item
 The Hilbert function for $R_w$ is given by 
\begin{equation*} h(d)
 = \sum_{
 \substack{J\subseteq [n]\\ |w_J| < |w|/2}} (-1)^{|J|} \binom{d\left(|w|/2-
|w_J|\right)+n-|J| -2}{n-2} \end{equation*}
if $d|w|$ is even, and zero otherwise. 

\item \label{cor:degree}
For $|w|$ even, the degree of $M_w$ is 

\[ \frac{1}{n-2}\left( \sum_{
\substack{J \subseteq [n]\\ |w_J|<|w|/2}}\left(-1\right)^{|J|}
\left(|w|/2 -|w_J|\right)^{n-3}\left(\sum_{i=0}^{n-3}
n-|J|-2-i\right)\right).
\]
\end{enumerate}
\end{theorem}

Let $K(\lambda, \mu)$ be the 
Kostka number counting semistandard Young tableaux of shape 
$\lambda$ with filling $\mu$.  
The dimension of the $d$-th graded part 
$(R_w)_d$ of $R_w$ is equal to 
the 
\emph{stretched} Kostka number $K(d\lambda,d\mu)$ where  
$\lambda = (|w|/2, |w|/2)$  
and  $\mu = w$ (see \ref{prop:ssytbasis}).  We give a closed 
formula for Kostka numbers of this form in Proposition~\ref{prop:Kostkapi}.
It was shown in \cite{KR86} and \cite{BGR04} that for 
partitions $\lambda$ and $\mu$ the function $K(d \lambda, d\mu)$ 
is a polynomial in $d$. 
Thus the Hilbert function gives a closed 
formula for the polynomials $K(d\lambda,d\mu)$ in this 
special case.

In \cite{1209.3689} Jakub Witaszek studies the multigraded Poincar{\'e}-Hilbert
series of the homogeneous coordinate ring of the Pl{\"u}cker embedding 
of the Grassmannian $G(2,n)$ 
for a certain $\N^n$-grading. We obtain a closed formula for the multigraded 
Hilbert function and for the Poincar{\'e}-Hilbert 
series, see Remark~\ref{rem:poincare}.

For a field $\k$, recall that a graded $\k$-algebra $R$ is \emph{Koszul} if 
$\k$ admits a linear free resolution as an $R$-module. 
If $R=\k[X_0, \ldots, X_N]/I$, then the existence of a quadratic 
Gr{\"o}bner basis for $I$ implies that $R$ is Koszul, which 
in turn implies that $I$ is generated by quadratic equations. 
In \cite{KeelTevelevKoszul}, Keel and Tevelev show that the section 
ring of the log-canonical line bundle on $\overline{M}_{0,n}$ 
is Koszul.  
However, while  for $w=1^8$, $I_w$ is generated 
by quadratic equations, we show in 
Example~\ref{ex:n8notKoszul} that $R_w$ is not Koszul. 

In general, high enough Veronese subrings of graded rings are Koszul
\cite{BackelinFroberg85, Koszul}, and up to a linear transformation, they 
admit a quadratic Gr{\"o}bner basis \cite{EisenbudReevesTotaro}.  
We show that for $R_w$ already the second Veronese subring 
satisfies these properties. 

\begin{theorem}\label{thm:main}
	Assume $w\in (2\Z)^n$. Then $I_w$ admits a squarefree quadratic 
	Gr{\"o}bner basis. In particular, $R_w$ is Koszul. 
\end{theorem}

Both theorems apply in all characteristics, and in fact more generally 
over the integers. Note that  
Theorem~\ref{thm:main} implies that for $w$  with $|w|$ odd, $I_w$ admits 
a squarefree quadratic Gr{\"o}bner basis, since  in this case $R_w=R_{2w}$  
by Proposition~\ref{prop:ssytbasis}. In particular, if $w=1^n$ with $n$ 
odd, then $I_w$ admits a quadratic Gr{\"o}bner basis. However, it is 
not known whether for $w=1^{10}$, $R_w$ is Koszul, see Remark~\ref{rem:Koszul?}.

As in \cite{HMSV, HMSV12}, our proof is based on a toric 
degeneration. This toric degeneration is 
a SAGBI degeneration, and we show that this toric degeneration admits a 
quadratic Gr{\"o}bner basis.  
After we shared our result with Manon, he was able to 
extend it to 
more general polytopes that arise as degenerations  
of the coordinate rings of the moduli stack of quasi-parabolic $\SL(2,\C)$ principal bundles on a generic marked projective curve in \cite[Theorem 1.10]{Manon12},
see 
Remark~\ref{rem:Manon}.

\begin{ack}
	We 
        benefited from discussions with many people, including  Federico 
	Ardila, Aldo Conca, Sergey Fomin, Nathan Ilten, Chris Manon, Sam Payne, Bernd Sturmfels, and Ravi Vakil. We would also like to thank the referee,  Diane Maclagan, and Burt 
	Totaro for 
helpful	comments on a previous version of this paper. Our main thanks is to 
	Vic Reiner who was shaping the direction 
	of this project. Part of this work was done at 
	the Institute of Mathematics and its Applications, at the Mathematisches Forschungsinstitut Oberwolfach, and at the Max-Planck-Institut f{\"u}r Mathematik, and we would like to thank these institutes 
	for providing a great research environment. 
\end{ack}

\section{The coordinate ring $R_w$ of $M_w$}\label{sec:prelims}
In this section we set up basic notation and describe the 
invariant ring $R_w$ in terms of certain semistandard Young tableaux.
Let $\k$ be a field, and 
let 
$$S = \k[x_1,y_1,x_2,y_2,\ldots,x_n,y_n],$$ which we
view as the set of polynomial functions on the space $\mathbb{A}^{2 \times n}$ of $2 \times n$ matrices with entries in the field $\k$:
$$\left( \begin{matrix} x_1 & \cdots & x_n \\ y_1 & \cdots & y_n
\end{matrix} \right).$$
The polynomial ring $S$ is  
graded by $\N^n$, where the degree of the monomial 
$\prod_{i=1}^n x_i^{a_i} y_i^{b_i}$ is equal to 
$(a_1 + b_1, a_2 + b_2, \ldots, a_n+b_n)$.  
Given $\br = (r_1, \ldots, r_n) \in \N^n$, let $S_{\br}$ denote the $\br$-th 
graded part of $S$.  Let $L_{\br} = \mathcal{O} (r_1, \ldots, r_n)$. 
Viewing $x_i,y_i$ as homogeneous coordinates for the $i$'th point 
in $(\Pj^1)^n$, we have 
$H^0\left(\left(\Pj^1\right)^n,L_{\br}\right)=S_{\br}$. 
The line bundle $L_r$ admits a linearization for the 
diagonal 
action of $\mathrm{SL}(2,\k)$ on $(\mathbb{P}^1)^n$,
(see \cite[Chapter 3.1]{MumfordGIT}) such that the 
induced action on the section ring $R(L_r)$ is given by 
matrix multiplication on the left. 
It is easy to see that for $i<j$  the polynomials 
\[ p_{ij} = \mathrm{det} \left(\begin{matrix} x_i & x_j \\ y_i& y_j\end{matrix}\right)\]
	are invariant under the $\mathrm{SL}(2,\k)$ action. 
	Note that $p_{i,j}$ are the Pl{\"u}cker 
	coordinates on the Grassmannian $G(2,n)$. 
	The First Fundamental Theorem of Invariant Theory 
says that 
they generate the ring of invariants $S^{\mathrm{SL}(2,\k)}$, \cite[Theorem 2.1]{Dolgachev}. Note that $S^{\mathrm{SL}(2,\k)}$ is the homogeneous coordinate 
ring of $G(2,n)$ in the Pl{\"u}cker embedding. 

For our purposes, it is most convenient  to study
this invariant ring using tableaux of shape $(k,k)$,  
\begin{equation}\label{eq:tau}
\tau = \tableau{i_1}{j_1} \tableau{i_2}{j_2} \cdots \tableau{i_k}{j_k},
\end{equation}
where $1\leq i_{\ell}, j_{\ell} \leq n$. 
A tableau  $\tau$ is called \emph{semistandard} if its entries are 
weakly increasing in the rows and strictly increasing in the 
columns, i.e., in our setting we have $i_1 \leq \cdots \leq i_k, j_1 \leq \cdots, \leq j_k$, 
and $i_1< j_1, \ldots, i_k < j_k$. 
The \emph{content} of a tableau $\tau$ is the $n$-tuple 
$w(\tau) = (w(\tau)_1, \ldots, w(\tau)_n)$, where $w(\tau)_i$ denotes 
the number of times $i$ occurs in $\tau$. 

To a semistandard tableau $\tau$ as in \eqref{eq:tau}
we associate the polynomial \[s_{\tau} = p_{i_1,j_1}p_{i_2,j_2}  \cdots 
p_{i_k,j_k}\in S^{\mathrm{SL}(2,\k)}.\] 
Note that $s_{\tau}$ is homogeneous of degree
$w(\tau)$.

The following proposition is an algebraic incarnation of 
the Gel'fand-MacPherson correspondence \cite{GelfandMacPherson}.
Let $R_w$ be the invariant ring of  the introduction and let 
$|w|=w_1+\cdots +w_n$.

\begin{proposition}\label{prop:ssytbasis}
	The polynomials $s_{\tau}$, where $\tau$ ranges over all 
	semistandard tableaux of shape 
	$(\frac{d|w|}{2},\frac{d|w|}{2})$
	with content $dw$ for some 
	$d\in \N$, form a vector 
	space basis for the invariant ring $R_w$.
\end{proposition}

\begin{proof}
	Let $X=(\Pj^1)^n$. 
The torus
$T_w = \{\mathrm{diag}(t_1,\ldots,t_n) \in \mathrm{GL}(n,\k) \mid t_1^{w_1} t_2^{w_2} \cdots t_n^{w_n} = 1\}$
acts on the right of $\mathbb{A}^{2 \times n}$  
by matrix multiplication inducing an action on $S$ such that 
$S^{T_w} = \oplus S_{dw}.$ 
Thus we obtain
\[R_w = \left( \bigoplus_d H^0\left(\left(\Pj^1\right)^n, L_w^d\right) 
\right)^{\mathrm{SL}(2,\k)} =  
\left( \bigoplus_d S_{dw}\right)^{\mathrm{SL}(2,\k)}
= \left(S^{T_w}\right)^{\mathrm{SL}(2,\k)}.\]

Since the actions of $\mathrm{SL}(2,\k)$ and $T_w$ commute, 
we have
$(S^{T_w})^{\mathrm{SL}(2,\k)} = (S^{\mathrm{SL}(2,\k)})^{T_w}$. 
Now, $S^{\mathrm{SL}(2,\k)}$ has a vector space basis 
consisting of $s_{\tau}$ where $\tau$ ranges 
over semistandard Young tableaux of shape $(k,k)$, see for example \cite[Theorem 2.3]{Dolgachev}. 
Let $f = \sum_{\tau} a_{\tau}s_{\tau} \in S^{\mathrm{SL}(2,\k)}$, where $\tau$ runs over semistandard tableaux. 
Since the action of $T_w$ is linear, and 
the $s_{\tau}$  are linearly independent, $f$ is invariant under $T_w$ if and only if every $s_{\tau}$ is invariant. 
Note that for a tableau $\tau$ with content
$w(\tau)$, we have $t \cdot s_{\tau} = t_1^{w(\tau)_1}\cdots t_n^{w(\tau)_n}s_{\tau}$. 
In particular, $s_{\tau}$ is 
invariant if and only if 
there is $d$ with $w(\tau) = dw$. 
The claim follows. 
\end{proof}

In particular, when $w=1^n$ and $n=2m$ even, the dimension of the space 
of lowest degree invariants in $R_w$ is the Catalan number $C_m$.

\begin{remark}
\label{rem:Veronesesub}
One can view $R_w$ as a multigraded  Veronese subring of the Pl{\"u}cker 
algebra. 
The Pl{\"u}cker algebra $S^{\mathrm{SL}(2,\k)}$ admits a 
$\N^n$-grading
determined by the weight 
under the action of the diagonal 
torus $T=\mathrm{diag}(t_1, \ldots, t_n)$ 
via $t\cdot p_{i,j} = t_it_jp_{i,j}$. Then for a tableau $\tau$ with 
content $w(\tau)$, we have $t\cdot s_{\tau} = t^{w(\tau)}s_{\tau}$ and thus
we can conclude as in the proof of Proposition~\ref{prop:ssytbasis} that
\[\left(S^{\mathrm{SL}(2,\k)}\right)_w 
= \langle\{s_{\tau} \mid \tau \textrm{ is semistandard  of shape }(|w|/2,|w|/2)\textrm{ with filling }w(\tau) = w\}\rangle.\] Here $\langle \cdots \rangle$ denotes the span as a vector space over $\k$. In particular, 
$\left(S^{\mathrm{SL}(2,\k)}\right)_w = 0$ if $|w|$ is odd. It then follows 
that 
$R_w = \bigoplus _{d\in \N} \left( S^{\mathrm{SL}(2,\k)}\right)_{dw}$. 
\end{remark}

\begin{definition}\label{def:partition}
	To  a tableau $\tau$  of shape $(k,k)$
	is associated a partition 
	$\nu_{\tau} = (\nu_1, \ldots, \nu_n)$ of $d$,  
	the content of the first row of $\tau$, i.e.,  
	$\nu_i = |\{j \mid i_j = i\}|$. 
\end{definition}
The following Lemma can be easily deduced from the 
discussion in \cite[Section 3]{HMSV05}. We include a sketch of the 
 proof for the 
convenience of the reader. 

\begin{lemma}\label{lem:Kostkafnu}
	The semistandard Young tableaux of shape $(|w|/2,|w|/2)$
	with filling $w$ 
	are in bijection with  partitions 
	$\nu=(\nu_1, \ldots, \nu_n)$ of $|w|/2$  satisfying 
	\begin{align}
		0\leq \nu_{\ell} &\leq w_{\ell} \textrm{ and} \label{eq:3}	\\
		2(\nu_1 + \cdots + \nu_{\ell-1}) + \nu_{\ell} & \geq w_1 + \cdots + w_{\ell} \label{eq:4}
	\end{align}
	for $1\leq \ell \leq n$. 
These conditions imply  $\nu_1 = w_1$ and $\nu_n =0$.
\end{lemma}

\begin{proof}
        To a tableau with filling $w$ and increasing rows
	is associated a partition of 
	$|w|/2$ by Definition~\ref{def:partition}.
	Conversely,
	to a partition $\nu$ satisfying \eqref{eq:3},
	we associate a tableau
	$\tau$  of shape $(\frac{|w|}{2},\frac{|w|}{2})$
	by filling the first row with $\nu_1$ 1's, $\nu_2$ 2's, 
	etc., and the second row with $(w_1-\nu_1)$ 1's, $(w_2 -\nu_2)$
	2's, etc. 
	By construction the rows 
	of this tableau are increasing and it has filling 
	$w$. These associations are inverse to each other.
	Moreoever, $\tau$ is 
	semistandard if and only if $\nu_1 + \cdots + \nu_{\ell-1}
	\geq (w_1 - \nu_1) + \cdots + (w_{\ell}-\nu_{\ell}) $ for
$1\leq \ell\leq n$. This condition is equivalent to 
\eqref{eq:4}. \end{proof}

\section{The Hilbert polynomial and degree of $M_w$}
In this section we will prove the formulas for the Hilbert function 
of $R_w$ and the degree of $M_w$ of Theorem~\ref{thm:hilbertfn}.
Our techniques are similar to those of Howe
\cite[5.4.2.3.]{Howe88} who computed the case when 
$w=1^n$.

Let
$\lambda = (|w|/2, |w|/2)$
and $\mu=w$. 
For partitions $\lambda$ and $\mu$ the Kostka 
numbers $K(\lambda, \mu)$ are defined 
to be the number of semistandard Young tableaux 
of shape $\lambda$ and content $\mu$. For a partition 
$\lambda = (\lambda_1, \ldots, \lambda_s)$, we let 
$d\lambda = (d\lambda_1, \ldots, d\lambda_s)$. 
Note that  by Proposition \ref{prop:ssytbasis} 
the dimension of $(R_{w})_d$ is equal to 
$K(d\lambda, d\mu)$.

In order to compute 
the Hilbert polynomial, we give a formula for these 
particular Kostka numbers.
The main step in proving this formula is to set up a relationship 
between the Kostka numbers and numbers of the corresponding 
partitions of Definition~\ref{def:partition}. 
We let  
$\Pi(n,\infty,k)=\{(\nu_1, \ldots, \nu_n) \mid 0\leq \nu_i \textrm{ for } i\in [n] \textrm{ and } \nu_1 + \cdots + \nu_n=k\}$,
and let 
\begin{equation}\label{eq:pi}
\pi(n,\infty,k) = |\Pi(n,\infty,k)| = \binom{n-1+k}{n-1}.
\end{equation}
Let
$w =(w_1, \ldots, w_n) \in \N^n$. We let 
$\Pi(n,w,k)=\{(\nu_1, \ldots, \nu_n) 
\mid 0\leq \nu_i \leq w_i \textrm{ for } i\in [n] \textrm{ and } \nu_1+\cdots + \nu_n = k\}$ 
and let $\pi(n,w,k) = 
|\Pi(n,w,k)|$. 
For a subset $I\subseteq [n]=\{1, \ldots, n\}$, we let 
$\Pi(n,w_I,k)
=\{(\nu_1, \ldots, \nu_n) \in \Pi(n,\infty,k) \mid 0\leq \nu_i \leq w_i 
\textrm{ for } i\in I\}$ 
and $\pi(n,w_I,k) = |\Pi(n,w_I,k)|$. 

\begin{lemma}\label{lem:pi} 
For any $I\subseteq [n]$, 
\begin{eqnarray*} 
\pi(n,w_I,k) &=& \sum_{J\subseteq I}(-1)^{|J|}\pi\(n,\infty,k-\(|w_{J}|+|J|\)\)\\ 
&=& \sum_{J\subseteq I} (-1)^{|J|}\binom{n-1+k-\(|w_{J}|+|J|\)}{n-1}.
\end{eqnarray*}
\end{lemma}

\begin{proof} 
We proceed by induction on the cardinality of $I$. When 
$I=\emptyset$, the above claim is immediate. If $I\neq \emptyset$, let 
$j\in I$. Since
\[ \pi(n,w_I,k) = \pi(n,w_{I\smallsetminus \{j\}},k)
- \pi(n,w_{I\smallsetminus \{j\}},k-w_j-1),\]    
the claim follows from the induction hypothesis. The last equality 
follows from (\ref{eq:pi}).  \end{proof}

\begin{proposition}\label{prop:Kostkapi}
	Let $w \in \N^n$ and assume that $|w|$ is even. Then for $\lambda = (|w|/2, 
|w|/{2})$ and $\mu = w$, we have 
\begin{eqnarray*}
K(\lambda, \mu) &=& 
\pi\left(n,w,|w|/2\right)-
\pi\left(n,w,|w|/2-1\right)\\
&=& 
 \sum_{
 \substack{J\subseteq [n]\\ |w_J| < |w|/2}} (-1)^{|J|} \binom{|w|/2-
|w_J|+n-|J| -2}{n-2} 
\end{eqnarray*}
\end{proposition}

\begin{proof}
For the first equality, we need to express the Kostka numbers in terms of 
partitions.  For  $\nu \in \R^n$, we define a function 
$f_{\nu} \colon [n] \to \R$ by 
\begin{equation}\label{eq:fnu}
 f_{\nu}(i) = 
2\nu_1 + \cdots + 2 \nu_{i-1} + \nu_i -(w_1 + 
\cdots + w_i) 
\end{equation}
for  $1\leq i \leq n$.
Then Lemma \ref{lem:Kostkafnu} implies that for
$w\in \Z^n$ with $|w|$ even, 
$\lambda = (|w|/2, |w|/2)$, and $\mu = w$ we have 
\begin{equation}\label{eq:Kostkafnu}
 K(\lambda, \mu) = |\{\nu \in 
\Pi\left(n,w,|w|/2\right)\mid  f_{\nu}(i) \geq 0 
\text{ for all } 1\leq i\leq n\}|.
\end{equation}
For $\nu \in \Z^n$ we let 
$ m_{\nu} = \min\{f_{\nu}(j)\mid j\in [n]\}$ and $i_{\nu} = 
\max \{ j \mid f_{\nu}(j) = m_{\nu}\}$ and define 
\begin{eqnarray*}\phi \colon \{\nu \in \Pi\left(n,w,|w|/2\right) 
\mid \exists\ i \text{ such that } f_{\nu}(i)< 0\}  &\to& \Pi\left(n,w,|w|/2-1\right)
\\ 
(\nu_1, \ldots, \nu_n) &\mapsto& (\nu_1, \ldots, \nu_{i_{\nu}} - 1, 
\ldots, \nu_n).
\end{eqnarray*}
We claim that $\phi$ is well-defined and gives  a bijection. 
Note that the first equality then follows from 
 this claim together with \eqref{eq:Kostkafnu}.

The following equalities follow easily from the definition of $f_{\nu}$:
\begin{eqnarray}
\label{eq:fnu2} f_{\nu}(i+1) & = & f_{\nu}(i) +\nu_i - w_{i+1} + \nu_{i+1}\\
\label{eq:fnu1} f_{\nu}(n) & = &2|\nu| - |w| - \nu_n. 
\end{eqnarray}
To see that $\phi$ is well defined, we have to show that 
$\nu_{i_{\nu}} > 0$. When $i_{\nu} <n$  then $f_{\nu}(i+1)
> f_{\nu}(i)$ and (\ref{eq:fnu2}) implies $\nu_i > 
w_{i+1}-\nu_{i+1} \geq 0$. If $i_{\nu} = n$, then 
$f_{\nu}(n) = m_{\nu} < 0$, since there exists $i$ such that 
$f_{\nu}(i)<0$ and $m_{\nu}\leq f_{\nu}(i)$.  Since 
$|\nu| = |w|/2$, 
$f_{\nu}(n) = 
-\nu_n$ by (\ref{eq:fnu1}), so we have $\nu_n > 0$.

To see that $\phi$ is a bijection, we exhibit the inverse map. 
For $\nu' \in \Pi\left(n,w,|w|/2-1\right)$, we let 
 $j_{\nu'} = \min \{ i \mid f_{\nu'}(i)= m_{\nu'}\}$ and define 
\[ \psi \colon (\nu'_1, \ldots, \nu'_n) 
\mapsto (\nu'_1, \ldots, \nu'_{j_{\nu'}} + 1, \ldots \nu'_n).\]
We have
\begin{equation}\label{eq:fnu'}
f_{\psi(\nu')}(i) = 
\left\{ \begin{array}{lcll}
f_{\nu'}(i) &\geq& m_{\nu'}+1& \text{ if } i<j_{\nu'} \\
f_{\nu'}(i) + 1& =& m_{\nu'} +1 & \text{ if } i = j_{\nu'} \\
f_{\nu'}(i) + 2& \geq &m_{\nu'}+2 & \text{ if } i > j_{\nu'}.
\end{array}
\right.
\end{equation}

We have to show that $\psi$ is well-defined. 
If $j_{\nu'}>1$, then $f_{\psi(\nu')}(j_{\nu'}) = 
f_{\psi(\nu')}(j_{\nu'}-1) + \nu'_{j_{\nu'}-1} - w_{j_{\nu'}}
+ \nu'_{j_{\nu'}} +1$ by \eqref{eq:fnu2}. Plugging in the 
values from \eqref{eq:fnu'}, we see that $w_{j_{\nu'}} \geq 
\nu_{j_{\nu'}}+1$.
If $j_{\nu'}=1$, then $f_{\psi(\nu')}(1) = \nu'_1 + 1 -w_1 
= m_{\nu'} +1$. However, note that 
$m_{\nu'} \leq f_{\nu'}(n) \leq -2$ by \ref{eq:fnu1}, 
and so it follows that $\nu'_1 +2 \leq w_1$. 
Moreover, 
(\ref{eq:fnu'}) implies that $i_{\psi(\nu')} = j_{\nu'}$. 
One can check similarly that $j_{\phi(\nu)} = i_{\nu}$ and 
it follows that  $\phi$ and $\psi$ are inverse 
to each other.  

For the second equality, note that Lemma~\ref{lem:pi} implies that 
\begin{multline*}
\pi\left(n,w,|w|/2\right)-
\pi\left(n,w,|w|/2-1\right) \\
=\sum_{J\subseteq [n]}(-1)^{|J|}\left[\binom{n-1+|w|/2-(|w_{J}|+|J|)}{n-1}
-\binom{n-2+|w|/2-(|w_{J}|+|J|)}{n-1}
\right]
\end{multline*}
Using the identity $\binom{m}{n} - \binom{m-1}{n} = 
\binom{m-1}{n-1}$, one obtains the formula in the statement.  
Note that if $|w_J| \geq |w|/2$, the expression in the top of the binomial 
coefficient is less than $n-2$, so 
it suffices to sum over those $J\subseteq [n]$ 
such that $|w_J| < |w|/2$. 
\end{proof}

\begin{proof}[Proof of Theorem \ref{thm:hilbertfn}]
	Fix $w\in \N^n$, let
$\lambda = (|w|/2, |w|/2)$, and let $\mu = w$. 
It follows from Proposition~\ref{prop:ssytbasis} that 
$\dim (R_w)_d=0$ if $d|w|$ is odd and that 
$\dim (R_w)_d 
= K(d\lambda, d\mu)$ if $d|w|$ is even.
The 
formula  for $h(d)$ then follows from Proposition~\ref{prop:Kostkapi}. 

The formula for the degree of $M_w$ is obtained by 
computing the coefficient of $d^{n-3}$ in the Hilbert polynomial and 
multiplying by $(n-3)!$. 
\end{proof}
\begin{remark}\label{rem:poincare}
Our formula also implies a closed formula for the multigraded Hilbert function 
and Poincar{\'e}-Hilbert 
series of the coordinate ring of the Grassmannian $G(2,n)$ in the 
Pl{\"u}cker embedding with the 
multigrading described in Remark~\ref{rem:Veronesesub}. 
In \cite[Theorem 3.4.3]{1209.3689} Jakub Witaszek 
gives a recursive formula for the 
multigraded Poincar{\'e}-Hilbert series  
$\sum_{w\in\Lambda} \mathrm{dim}\left(S^{\mathrm{SL(2,\k)}}\right)_wz^w$
of the Pl{\"u}cker algebra for the multigrading described 
in Remark~\ref{rem:Veronesesub}.
He also obtains a combinatorial formula  for the Poincar{\'e}-Hilbert series.

Let 
$\Lambda=\{w \in \Z^n \mid |w|\in 2\Z\}$. 
By Remark~\ref{rem:Veronesesub}, we have  $\mathrm{dim}\left(S^{\mathrm{SL(2,\k)}}\right)_w = 
K(\lambda, \mu)$, where $\lambda = (|w|/2,|w|/2)$ and $\mu=w$. 
It follows that the support of the multigraded Hilbert function $\mathbf{h}$ 
is $\{ w\in \Lambda \mid |w|/2 \geq w_j \textrm{ for all } 1\leq j\leq n\}$. 
Then Proposition~\ref{prop:Kostkapi}  implies 
that for $w\in \Lambda$, 
\begin{equation*}
 \mathbf{h}(w) =\mathrm{dim}\left(S^{\mathrm{SL(2,\k)}}\right)_w =
\sum_{\substack{J\subseteq [n]\\|w_J|<|w|/2}} (-1)^{|J|} \binom{|w|/2-|w_J|+n-|J|-2}{n-2}.
\end{equation*}
So we obtain a closed formula for the multigraded Poincar{\'e}-Hilbert series.

A point $p=(p_1, \ldots, p_n) \in (\Pj^1)^n$ is stable (resp. semistable) for the $\mathrm{SL}(2,\k)$-linearization of $L_w$ if for all subsets of 
indices of colliding points $J=\{j\in [n] \mid p_j =p \textrm{ for some } p\}$ 
we have $|w_J| 
< |w|/2$ (resp. $|w_J|\leq |w|/2$),
see 
\cite[Chapter 3]{MumfordGIT}, 
\cite[Section 6]{Thaddeus96}, or \cite[Section 8]{Hassett03}.
Identifying $\N^n$ with the effective divisors on 
$(\Pj^1)^n$, we see that the fact that in the formula for the  multigraded Hilbert function 
we sum over those $J\subseteq [n]$ 
such that $|w_J| < |w|/2$ reflects the chamber structure for the 
GIT chambers whose walls are given by $|w_J|=|w|/2$.
In particular, the multigraded Hilbert function  is piecewise polynomial in 
$w\in \Lambda$ and the domains of polynomiality agree with the GIT chambers. 
\end{remark}

\begin{remark}
In the formula for the Hilbert polynomial of Theorem 
\ref{thm:hilbertfn}, 
the terms of degree $(n-2)$ cancel out since 
$\dim M_w = n-3$. Thus we obtain the following identity:
\[ 
\sum_{J\subseteq [n]} (-1)^{|J|} (|w|/2-|w_J|)^{n-2} = 0.
\]
\end{remark}

\begin{remark}
Let $g(d)$ denote the Hilbert function of $G(2,n)$ in the Pl{\"u}cker embedding. Recall the multigraded 
Hilbert function $\mathbf{h}$ for the Pl{\"u}cker embedding from 
Remark~\ref{rem:poincare}. Then we have 
\[ g(d) = \sum _{\substack{w \in \N^n \\ |w|=2d}} \mathbf{h}(w),\]
implying the identity 
\begin{multline}
\binom{n+d-1}{d}^2 - \binom{n+d}{d+1}\binom{n+d-2}{d-1} 
\\ = \sum _{\substack{w \in \N^n \\ |w|=2d}}
\sum_{\substack{J\subseteq [n]\\|w_J|<|w|/2}} (-1)^{|J|} \binom{|w|/2-|w_J|+n-|J|-2}{n-2}.
\end{multline}
\end{remark}

\begin{remark}
	While our formula counts semistandard tableaux, 
	it contains negative signs. It would be nice to 
	have a formula with positive coefficients. In fact, 
	King, Tollu and Toumazet 
	conjecture that for arbitrary $\lambda,\mu$ 
	the coefficients of the polynomial $K(d\lambda, d\mu)$ are 
	positive in \cite[Conjecture 3.2]{KTT}.
\end{remark}

\begin{remark}
 Narayanan shows in \cite[Theorem 1]{Narayanan} that the problem of computing Kostka 
 numbers $K(\lambda,\mu)$ is $\#P$-complete. 
 Note that for our formula, one has to compute first all subsets of $[n]$, 
 where $n$ is the length of $\mu$. 
\end{remark}

\begin{example}
	The formula in Theorem~\ref{thm:hilbertfn} shows that $\mathrm{deg}(M_4) =1$, 
	$\mathrm{deg}(M_6) = 3$, $\mathrm{deg}(M_8) = 40$, $\mathrm{deg}(M_{10}) = 1225$, $\mathrm{deg}(M_{12})=67956$, $\mathrm{deg}(M_{14})=5986134$, $\mathrm{deg}(M_{16}) = 769550496$,  
 so  
	this sequence agrees with  A012250 on Sloane's online encyclopedia of integer sequences \cite{sloaneA012250}, compare \cite[Section 2.15]{HMSV}. 
	Similarly, when  
	$w = (2,\ldots, 2)$, the degrees of $M_w$ 
	are  $\mathrm{deg}(M_{2^4}) = 2$, $\mathrm{deg}(M_{2^5})= 5$,
	$\mathrm{deg}(M_{2^6})=24$, $\mathrm{deg}(M_{2^7})=154$, $\mathrm{deg}(M_{2^8})=1280$,  $\mathrm{deg}(M_{2^9}) = 13005$, $\mathrm{deg}(M_{2^{10}})=156800$, $\mathrm{deg}(M_{2^{11}})=2189726$ which agrees with 
	sequence A012249 \cite{sloaneA012249}. 
\end{example}

The following example shows that while for $w=1^8$, 
$I_w$ is generated 
by quadratic equations by \cite{HMSV12}, the ring of invariants
$R_w$ is not Koszul. 

\begin{example}\label{ex:n8notKoszul}
	Recall that for a graded algebra $R$ the Hilbert series is given by $H(z) = \sum_{d=0}^{\infty}\mathrm{dim}(R_d)z^d$. Similarly, the Poincar{\'e} series is given  by 
	$P(z) = \sum_{i=0}^{\infty} \mathrm{dim}\mathrm{Tor}_i^{R}(\k,\k)z^i$. Let $\mathbf{P}(u,v) = \sum_{i=0}^{\infty} \mathrm{dim}\mathrm{Tor}_i^{R}(\k,\k)_ju^iv^j$.
	Then we have $H(z)\mathbf{P}(-1,z)=1$ by 
	\cite[Chapter 2, Proposition 2.1]{PolishchukPositselski}. 
	If $R$ is Koszul, then the 
	minimal free graded resolution of $\k$ over $R$ is linear,
	and so $P(uv) = \mathbf{P}(u,v)$. Then 
	$H(z)P(-z)=1$ if and only if $R$ is Koszul, see \cite[Theorem 1]{Froeberg99}.
In particular, the 
power series $H(-z)^{-1}=P(z)$ must have positive coefficients. 

	Note that for $w=1^8$, the Hilbert function 
	is given by 
	\[ h(d) = \sum_{j=0}^{3} (-1)^j \binom{8}{j} \binom{d(4-j)+6-j}{6},\]
	so the Hilbert series is 
	\begin{align*}
		H(z)&= 1 + 14z + 91z^2  + 364z^3  + 1085z^4  + 2666z^5  
		+ 5719z^6 \\ &\qquad+ 11096z^7 + 19929z^8 + O(z^9) \\
		&= \frac{1+8z+22z^2+8z^3+z^4}{(1-z)^6}.
	\end{align*}
	Then
	\begin{multline}
		H(-z)^{-1} = 1 + 14z + 105z^2  + 560z^3  + 2296z^4  + 
	6880z^5
	+8904z^6 \\ - 62320z^7 - 641704z^8 + O(z^9). \end{multline}

\end{example}

\begin{remark}\label{rem:Koszul?}
For $w=1^{10}$ it is not known whether $R_w$ is Koszul or whether the ideal of relations between the generators admits a quadratic Gr{\"o}bner basis. One can check that the first 800 coefficients of $H(-z)^{-1}$ are positive, however it is not known whether the relation $H(-z)P(z)=1$ is satisfied. In general one might hope that for $n$ large and $w=1^n$ the ring $R_w$ is Koszul or has other nice properties, such as Green's property $N_p$, see for example \cite[Section 1.8.D]{Lazarsfeld04a}. 
\end{remark}

\section{A SAGBI degeneration of $R_w$}

A crucial part in the proof of \cite[Theorem 1.1]{HMSV12} is 
the existence of toric degenerations of $M_w$ indexed by trivalent 
trees on $n$ leaves, see \cite[Section 3.3]{HMSV12}. 
The existence of one of these toric degenerations had been established by Foth and Hu in 
\cite[Theorem 3.2]{FothHu05}. 
In fact, the toric degenerations are torus quotients of the 
toric degenerations of the Grassmannian $G(2,n)$ studied by 
Sturmfels and Speyer  \cite{SpeyerSturmfels04} and by Gonciulea and Lakshmibai 
in \cite{GonciuleaLakshmibai96}. 
Manon mentions in \cite[Theorem 1.3.6]{Manon09} that 
the ring of invariants $R$ 
admits a  SAGBI degeneration.
In this section we will give an explicit description of this SAGBI 
degeneration, by taking the torus invariants of the SAGBI 
degeneration described in \cite[Section 14.3]{MillerSturmfels}.

Let $R$ be a finitely generated subalgebra of the polynomial ring $S=\k[x_1, \ldots, x_n]$, and let $\prec$
be a term order on $S$. Let $\inn_{\prec}(R)$ be the subalgebra 
of $S$ generated by the inital terms of the elements of $R$. 
Assume this subalgebra is finitely generated. (Note that 
this is rarely the case). A set of generators 
$\{f_1, \ldots, f_r\}$ is called a \emph{SAGBI basis} for $R$ with 
respect to $\prec$ if $\inn_{\prec}{f_1}, \ldots, \inn_{\prec}{f_r}$ generate 
$\inn_{\prec}(R)$. Let $I$ be the ideal of relations between 
the generators of $R$. Then $\inn_{\prec}(R) \cong R/\inn_{\prec}I$, 
see \cite[Proof of Corollary 2.1]{ConcaHerzogValla}.
Since $\inn_{\prec}(R)$ is a monomial algebra, it is toric, 
and therefore the existence of a SAGBI basis for an algebra implies the existence 
of a flat degeneration of this algebra to a toric algebra \cite[Theorem 15.17]{Eisenbud95}. 
Note that the toric algebra need not be normal. 

Now let $S=\k[x_1, \ldots, x_n,y_1, \ldots, y_n]$ and $R$ the 
ring of invariants as in Section \ref{sec:prelims}.
Let $\prec$ be the  the purely 
lexicographic term order with $x_1 \succ \cdots \succ x_n \succ y_1 \succ 
\cdots \succ y_n$. 

Recall  the polynomial $s_{\tau}$ associated to a $2\times n$ tableau $\tau = \tableau{i_1}{j_1} 
\tableau{i_2}{j_2} \cdots \tableau{i_r}{j_r}$. 
We also  associate a 
monomial $m_{\tau} = x_{i_1}\cdots x_{i_r}y_{j_1}\cdots 
y_{j_r} \in S$. 

\begin{lemma}\label{lem:lm}
	A monomial $m$ is a leading monomial of an element in $R$
	if and only if $m=m_{\tau}$, where $\tau$ is semistandard 
	with filling $dw$, for some $d$. In particular, the set of these $m_{\tau}$ 
 is a vector space basis for 
	$\inn_{\prec}(R)$.  
\end{lemma}

\begin{proof}
	By Proposition \ref{prop:ssytbasis}, the polynomials 
	$s_{\tau}$, where $\tau$ ranges over semistandard Young 
	tableaux of shape $2 \times \frac{d|w|}{2}$ with content
	$dw$ form a vector space basis of $R$. 
Note that every monomial occurring in $s_{\tau}$ is of the 
form $m_{\tau'}$, where $\tau'$ is a (not necessarily semistandard)
tableau with the  same shape and same 
content as $\tau$. Among those monomials, the largest 
with respect to the term order $\prec$ is $m_{\tau}$. In particular, 
$m_{\tau} \in \inn_{\prec}(R)$.  Moreover, the leading monomial of an element $\sum_{\tau \in T} a_{\tau} s_{\tau} \in R$ is
$m_{\tau'}$, where $m_{\tau'}$ is the largest monomial 
in $\{m_{\tau} \mid \tau \in T\, a_{\tau} \neq 0\}$ with respect to the 
term order $\prec$.  
Since to distinct semistandard tableaux are associated distinct monomials, 
the $m_{\tau}$
are also linearly independent.
See also  
       \cite[Lemma 14.13]{MillerSturmfels}. 
\end{proof}

\begin{definition}
	Let $w\in \N^n$. Let $Q_w \subset \R^n$ be the 
	polytope  defined by $\nu_1+ \cdots + \nu_n = \frac{|w|}{2}$
	and the inequalities
	\eqref{eq:3} and \eqref{eq:4}.
\end{definition}
\begin{example}\label{ex:Qw}
	Let $w=(2,2,2,2,2)$. Then the inequalities for 
	$Q_w$ imply $\nu_1=2, \nu_5=0$, and $\nu_4 = 3-\nu_2 -\nu_3$. Thus $Q_w$ is the 2-dimensional polytope given by the inequalities
	$0\leq \nu_2, \nu_3 \leq 2, 0\leq 3- \nu_2-\nu_3 \leq 2$, and $2\nu_2+\nu_3\geq 2$. From this we can see that  $Q_w$ is isomorphic to the convex hull of  
	$\langle (1,0),(2,0),(2,1),(1,2),(0,2)\rangle$. 
	
	\begin{figure}[htp]\label{fig:1}
		  \begin{minipage}{\textwidth}
			      \centering
			      \begin{tikzpicture}[scale=2]
					  \definecolor{fillColor}{gray}{0.9} 
					  \path[draw=black,thick,fill=fillColor] (1,0) -- (2,0) -- (2,1) -- (1,2) --(0,2) -- (1,0);
					  \draw[step=1cm] (-.1,-.1) grid (2.1,2.1);
					  \foreach \x in {0,1,2}{
					        \foreach \y in {0,1,2}{ 
						        \node[draw,circle,inner sep=2pt,fill] at (\x,\y) {};
								          } }
					  \path[draw=black,thick,fill=fillColor] (4,0) -- (5,1) -- (5,2) --(4,2) --(3,1)-- (4,0);
					  \draw[step=1cm] (2.9,-.1) grid (5.1,2.1);
					  \foreach \x in {3,4,5}{
					        \foreach \y in {0,1,2}{ 
						        \node[draw,circle,inner sep=2pt,fill] at (\x,\y) {};
								          } }
						    \end{tikzpicture}%
							\caption{The polytopes $Q_w$ of Example~\ref{ex:Qw} (with reference lattice $\Z^n$) and $P_w$  of Example~\ref{ex:Pw} (with reference lattice $(2\Z)^n$) for $w=(2,2,2,2,2)$.}
								        \end{minipage}
								\end{figure}
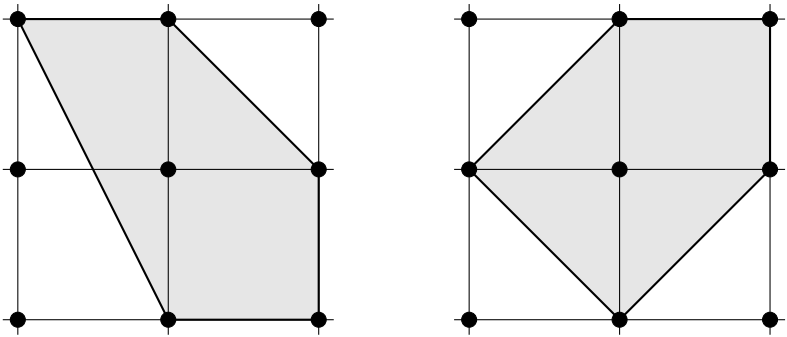
\end{example}

	Note that when $w=(1,1,1,1,1)$, $Q_w$ contains no lattice points; in particular, it is not a lattice polytope. However, it follows 
	from Lemma~\ref{lem:polytopeiso} and Lemma~\ref{lem:normal} that $2Q_w$ is 
	a lattice polytope for all $w$.

\begin{remark}
	Note that  $Q_w$ is a  Gelfand-Tsetlin polytope, see
	for example \cite{deLoeraMcAllister}.
\end{remark}

Recall that to a rational polytope $P$ is associated a 
graded monoid  $S_P=\{({\bf u},d)\mid  {\bf u}\in dP\cap \mathbb{Z}^n, d\in \N\}$ 
and an algebra $\k[S_P] = \langle x^{\bf{u}}z^d \mid ({\bf u},d) \in S_P\rangle$. 

\begin{proposition}\label{prop:sagbi}
        The subalgebra $\inn_{\prec}R$ 
	is isomorphic to 
        the polytopal 
	semigroup algebra $\k[S_{Q_w}]$. In particular, $\inn_{\prec}R$ is 
	finitely generated. 
\end{proposition}

\begin{proof}
	By Lemma \ref{lem:lm}, $\inn_{\prec}(R)$ is generated 
	as a vector space by $m_{\tau}$ where $\tau$ runs over all  
	semistandard Young tableaux of shape $2\times \frac{d|w|}{2}$ 
	with content $dw$. 
 	Note that for such a  semistandard Young tableau $\tau$ 
	we have $m_{\tau} = x^{\nu} y^{w-\nu}$, where $\nu$ is the partition 
	associated to $\tau$ as in Definition \ref{def:partition}.
	Let $\phi \colon \inn_{\prec} R\to \k[x_1, \ldots, x_n,z]$ be the 
	homomorphism induced by letting $\phi(m_{\tau}) = x^{\nu}z^d$, when 
        $\tau$ has shape $2\times \frac{d|w|}{2}$. 
        Since $\nu$ determines $\tau$, this homomorphism is 
	injective. It follows from 
	Lemma \ref{lem:Kostkafnu}  that it is surjective 
	onto $\k[S_{Q_w}]$. 
Note that for any rational polytope, the associated 
polytopal semigroup algebra $\k[S_{Q_w}]$ is finitely generated.
\end{proof}

\begin{remark}
When $w=1^n$, one can show that this toric degeneration is a 
degeneration of Fano varieties, and the corresponding line bundle the 
anticanonical line bundle. As this seems well known, we omit the proof. 
\end{remark}

\section{The quadratic Gr\"obner basis}

We now assume that $w\in (2\Z)^n$. The goal of this section is to show that in this case, the polytopal 
semigroup algebra $\k[S_{Q_w}]$ is generated in degree 1, and admits a presentation 
such that the ideal of relations has a quadratic Gr{\"o}bner basis. 
It then follows from general properties of SAGBI degenerations 
that $R_w$ also admits such a presentation. 

Instead of showing these properties directly for the polytopes 
$Q_w$, we will show them for a family of isomorphic polytopes $P_w$.
The latter ones exhibit more symmetry that we will exploit later on.
We denote the $i$'th component of a vector $u\in \R^{n-3}$ by $u(i+1)$ instead of $u_i$. 

\begin{definition}\label{def:P}
	We say that  a point $(x,y,z) \in \R^3$ \emph{satisfies the triangle inequalities}
	if $x+y \geq z, x+z \geq y$, and $y+z \geq x$. 
	To $w\in (2\N)^n$ is associated a polytope $P_w \subset \R^{n-3}$
	consisting of $(u(2), \ldots, u(n-2))$ such that 
	\[\left(w_1,w_2, u(2)\right), \left(u(n-2),w_{n-1},w_n\right), \textrm{ and  }
	\left(u(i-1),w_i,u(i)\right)\] satisfy the triangle inequalities for
        $3\leq i \leq n-2$. 
\end{definition}

\begin{example}\label{ex:Pw}
	When $w=(2,2,2,2,2)$, the polytope $P_w$ is given by the 
	inequalities $0\leq u(2)\leq 4$, $0\leq u(3) \leq 4$, 
	$u(2)+u(3) \geq 2, u(2)+2\geq u(3), u(3)+2 \geq u(2)$. See Figure~\ref{fig:1}.
\end{example}

It follows from Lemma~\ref{lem:normal} that $P_w$ is a lattice polytope for the 
lattice $M:=(2\Z)^{n-3}$.

\begin{lemma}\label{lem:polytopeiso}
	The polytopal semigroups $S_{Q_w} = \{(\nu,d) \mid d\in 
	\N, \nu\in dQ_w \cap \Z^{n}\}$ 
	and $S_{P_w} = \{(u,d) \mid d\in \N, 
	u\in dP_w \cap M\}$ are 
	isomorphic. 
\end{lemma}
\begin{proof}
Let $V = \{(\nu_1, \ldots, \nu_n, d) \in  \R^{n}\times \R \mid
\nu_1=dw_1, \nu_n=0,\textrm{ and } \nu_1+ \cdots + \nu_n = \frac{d|w|}{2}\}$, an affine subspace of $\R^{n}\times \R$. Then $S_{Q_w} \subset V$. We identify $V$ with $\R^{n-3}\times \R$  
via $(\nu_1, \ldots, \nu_n,d) \mapsto (\nu_2, \ldots, \nu_{n-2},d)$, with
inverse $(\nu_2, \ldots, \nu_{n-2},d) \mapsto (dw_1, \nu_2, \ldots, \nu_{n-2}, \frac{d|w|}{2}- dw_1 -\nu_2 - \cdots -\nu_{n-2}, 0, d)$. 
Since $|w|$ is even, this identification respects the lattices 
$V\cap (\Z^{n}\times \Z)$ and $\Z^{n-3}\times \Z$. 
	 Let \[\phi\colon \R^{n-3} \times \R\to \R^{n-3}\times \R, (\nu_2, \ldots, 
	\nu_{n-2},d) \mapsto (u(2), \ldots, u(n-2),d),\] 
	where $u(\ell) = 2(\nu_2 + \cdots +\nu_{\ell}) - d(w_2+\cdots +w_{\ell} -w_1)$ for $2\leq \ell \leq n-2$.
	Then $\phi$ has an  inverse  given by letting \[\nu_2 
	= \frac{u(2)-dw_1+dw_2}{2} 
	\textrm{ and }
	\nu_{\ell} = \frac{u(\ell)-u(\ell-1)+dw_{\ell}}{2}\]
        for $3\leq \ell \leq n-2$. So $\phi$ is an isomorphism. 
Moreover, it induces an isomorphism between $\Z^{n-3}\times \Z$ and $M\times \Z$. 

        We claim that   $\phi(dQ_w\times \{d\}) = dP_w\times 
	\{d\}$.  Using the fact 
	that for $(\nu_1, \ldots, \nu_n) \in dQ_w$ we have 
	$\nu_1=dw_1, \nu_n = 0$, and 
	$\nu_{n-1} = \frac{d|w|}{2} - (dw_1+\nu_2 + \cdots + \nu_{n-2})$, 
	the inequalities 
	for $dQ_w \cap \R^{n-3}$ are in the left  column  
	below, where $3\leq \ell \leq n-2$. 
	The corresponding inequalities for $dP_w$ are on the 
	right. 
	\begin{IEEEeqnarray*}{rClrCL}
		\nu_2 & \geq & 0  &\  u(2) & \geq & dw_1 - dw_2 \\
		 \nu_2 & \leq &dw_2  & u(2) &\leq &dw_1 + dw_2 
		\\
		\nu_2 & \geq & dw_2 -dw_1 &\  u(2) &\geq & dw_2 - dw_1 
		\\
	\nu_{\ell}&\geq & 0 & u(\ell -1) -u(\ell) &\leq& dw_{\ell} \\
		\nu_{\ell} &\leq & dw_{\ell} 
		& u(\ell) - u(\ell-1) &\leq & dw_{\ell} 
		\\
	2(\nu_2+\cdots +\nu_{\ell-1})+\nu_{\ell}& \geq &
	d(w_2+\cdots +w_{\ell}-w_1) 
	&\qquad  u(\ell) + u(\ell-1) &\geq & dw_{\ell} \\
		\nu_2 + \cdots + \nu_{n-2} &\leq & 
		\frac{d|w|}{2} - dw_1 
		& u(n-2) & \leq & dw_{n-1} + dw_n 
		\\
		\nu_2 + \cdots + \nu_{n-2} & \geq & \frac{d|w|}{2} -dw_1-dw_{n-1} 
	&	u(n-2) & \geq & dw_n-dw_{n-1}\\
		\nu_2+\cdots + \nu_{n-2} & \geq & \frac{d|w|}{2} - dw_1- dw_n 
		& u(n-2) & \geq & dw_{n-1} -dw_n 
	\end{IEEEeqnarray*}
 where the last inequality on the left follows from 
\eqref{eq:4} for $\ell=n-1$. 
	It is now easy to check that the inequalities on the left 
	correspond to the inequalities on the right under 
	$\phi$, so $\phi$ induces the required 
	isomorphism of semigroups. 
\end{proof}

\begin{lemma}\label{lem:normal}
	We have the following properties of $P_w$.
	\begin{enumerate}[(i)]
		\item \label{lem:normal1}	
The polytope $P_w$ is normal with respect to $M$, i.e., every lattice point in $mP_w\cap M$ is a sum of $m$ lattice points in $P_w\cap M$. 
\item \label{lem:normal2} For $v,v'\in P_w\cap M$ there is $u,u'\in 
	P_w\cap M$ such that $v+v' = u+u'$ 
	and $|u(i)-u'(i)|\leq 2$ for all $2\leq i\leq n-2$. 
\end{enumerate}
\end{lemma}

\begin{proof}
The proof of (i) closely follows \cite[Lemma 7.3]{HMSV} and is essentially the proof of  Lemma 6.4 in \cite{HMSV05}.
We first need to introduce some notation. 
Let $\sigma \in \{+,-\}$ and let $e^{\sigma}$ 
denote rounding to the nearest even integer, where for $a \in \Z$, 
we let $e^{+}\left(2a+1\right) = 2a+2$ and $e^{-}\left(2a+1\right) = 2a$. 
For  
$r = (r(2),\ldots,r(n-2)) \in mP_w \cap M$, 
we say that a sequence of signs  
$\sigma(i)\in \{+,-\}$  for $2\leq i \leq n-3$ is $(r,m)$-admissible 
if it satisfies 
$\sigma\left(i+1\right) = - \sigma \left(i\right)$ if and only if  
$\frac{r\left(i\right)}{m}$ and 
$\frac{r\left(i+1\right)}{m}$ are odd integers 
and $\frac{r\left(i\right)}{m}+\frac{r\left(i+1\right)}{m} = 
w_{i+1}$. Such a sequence exists and is unique up to a global sign change. 

We claim that if 
$\sigma(i)$ is $(r,m)$-admissible, then
\[u_r = 
\(u\left(2\right),\ldots,u(n-2)\) =
\(e^{\sigma(2)}\(\frac{r(2)}{m}\), \ldots,e^{\sigma(n-2)}\(\frac{r(n-2)}{m}\)\), \]
is a  lattice point in  $P_w$.  

The following properties of $e^{\sigma}$ for $\sigma 
\in \{+,-\}$, $a\in 2\Z$ and $x,y\in \bR$ will be useful:
\begin{enumerate}
\item \label{eq:propI} $e^{\sigma}$ is increasing. 
\item \label{eq:propII} $e^{\sigma}\left(x+a\right)
=e^{\sigma}\left(x\right) +a.$ 
\item \label{eq:propIV} If $a\geq x$ then $a\geq e^{\sigma}(x)$.
\item \label{eq:propIII} If $x+y\geq a$  then $e^{\sigma}\left(x\right) +e^{-\sigma}\left(y\right)\geq a.$
\item \label{eq:propV} $e^{\sigma}(x) + e^{\sigma}(y) 
\geq x+y-2$. 
\item \label{eq:propVI} $e^{\sigma}(-x) =-e^{-\sigma}(x)$.

\end{enumerate}

That $\left(w_1,w_2,u\left(2\right)\right)$ and 
$\left(u\left(n-2\right),w_{n-1},w_n\right)$ satisfy the triangle 
inequalities follows 
from the assumption that $\left(w_1,w_2,\frac{r\left(2\right)}{m}\right)$ and 
$\left(\frac{r\left(n-2\right)}{m},w_{n-1},w_n\right)$ satisfy the 
triangle inequalities and 
(\ref{eq:propI}), (\ref{eq:propII}) and (\ref{eq:propIV}). 
For example, 
we have $w_1 + u\left(2\right)= w_1+e^{\sigma}\left(\frac{r\left(2\right)}{m}\right) 
= e^{\sigma}\left(w_1 + \frac{r(2)}{m}\right) \geq 
e^{\sigma}\left(w_2\right) = w_2$ and $w_1+w_2 \geq 
e^{\sigma}\left(\frac{r(2)}{m}\right)=u(2)$.

When $2\leq i\leq n-3$, we have to show 
\begin{align}
\label{eq:ineq2} u\left(i\right) + w_{i+1} &\geq u\left(i+1\right) \\
\label{eq:ineq1} u\left(i+1\right) + w_{i+1} &\geq u\left(i\right) \\
\label{eq:ineq3} u\left(i\right) + u\left(i+1\right) &\geq w_{i+1}.
\end{align}
We consider two cases. 
We first assume that $\sigma\left(i\right) = \sigma\left(i+1\right)$, and 
we let $\sigma = \sigma\left(i\right)=\sigma \left(i+1\right)$. 
Then $\frac{r\left(i\right)}{m}$ and $\frac{r\left(i+1\right)}{m}$ 
are not both odd integers 
or $\frac{r\left(i\right)}{m} + \frac{r\left(i+1\right)}{m} >  w_{i+1}$. 
To see (\ref{eq:ineq1}), note that $u\left(i+1\right) +w_{i+1} 
= e^{\sigma}\left(\frac{r\left(i+1\right)}{m}\right)+w_{i+1}
= e^{\sigma}\left(\frac{r\left(i+1\right)}{m}+w_{i+1}\right)
\geq e^{\sigma}\left(\frac{r\left(i\right)}{m}\right) = u\left(i\right)$ by (\ref{eq:propI}), (\ref{eq:propII})
and the assumption that $\left(\frac{r\left(i\right)}{m}, 
\frac{r\left(i+1\right)}{m}, w_{i+1}\right)$ satisfy the triangle 
inequalities. 
(\ref{eq:ineq2}) follows similarly. For (\ref{eq:ineq3}), 
if both $\frac{r\left(i\right)}{m}$ and $\frac{r\left(i+1\right)}{m}$ are odd  integers 
then $\roverm{i}+\roverm{i+1}> w_{i+1}$ by assumption. 
Hence, by (\ref{eq:propV}),  $u\left(i\right)+u\left(i+1\right)\geq \roverm{i}+\roverm{i+1} - 2 > w_{i+1}- 2$
which implies (\ref{eq:ineq3}) since $u\left(i\right),u\left(i+1\right)$ and 
$w_{i+1}$ are even. Otherwise there exists 
$x\in \left\{\frac{r\left(i\right)}{m},\frac{r\left(i+1\right)}{m}\right\}$ that is not an odd integer. Then $e^{\sigma}\left(x\right) = e^{-\sigma}\left(x\right)$ and now 
(\ref{eq:ineq3}) follows from 
\eqref{eq:propIII}.

Suppose now that 
$\sigma\left(i+1\right) 
 = -\sigma\left(i\right)$. Then 
$\roverm{i}$ and $\roverm{i+1}$ are odd integers
and $\roverm{i} + \roverm{i+1} = w_{i+1}$. 
Then \eqref{eq:ineq3} follows from \eqref{eq:propIII}. If $\sigma\left(i\right) = +$ and $\sigma\left(i+1\right) = -$ 
then \eqref{eq:ineq2} follows easily. 
Since $\roverm{i+1}$ is a positive odd integer, we have 
$2\roverm{i+1}>0$, and using 
the assumption $\roverm{i}+\roverm{i+1} = w_{i+1}$, we obtain 
$u\left(i+1\right) + w_{i+1} = \roverm{i+1} - 1 + w_{i+1} > \roverm{i}  - 1 = u\left(i\right) - 2$. Since $u(i+1), 
u(i)$ and $w_{i+1}$ are even integers, \eqref{eq:ineq1} follows. 
The case $\sigma\left(i\right) = -$ and 
$\sigma\left(i+1\right)= +$ follows analogously.  

To show (i), we proceed by induction on $m$. For $m=1$ there is nothing to show. 
Assume that $m\geq 2$.
For $r\in mP\cap M$, let $u=u_r$ as in the claim. 
Then by the claim, $u$ lies in  
$P_w \cap M$. Let $v = (v(2), \ldots, v(n-2))$, where $v(i)=e^{-\sigma(i)}\left(\frac{m-1}{m}r(i)\right)$. Note that it follows from \eqref{eq:propIII} and \eqref{eq:propVI} that  $u(i) + v(i) = 
r(i)$, so 
$u+v = r$. 
Replacing $u$ by $v$, $\sigma$ by $-\sigma$, $w_i$ by $(m-1)w_i$ and $\frac{r(i)}{m}$ by $\frac{m-1}{m}r(i)$, and noting that $\left(\frac{m-1}{m}r\right) \in (m-1)P_w$ and that $\frac{r(i)}{m}$ is odd if and only if $\frac{(m-1)r(i)}{m}$ is odd, 
the same arguments as in the proof of 
the claim show that $v=r-u\in (m-1)P_w\cap M$. But $v$ is a sum of $m-1$ lattice points in $P_w$ by induction. 

For (ii),
we apply the claim to 
$r=v+v'$ and
let $(\sigma(2), \ldots, \sigma(n-2))$ be a
$(r,2)$-admissible sequence of signs. 
Then by  
the 
claim we have that 
\begin{eqnarray*} u &=&
\(e^{\sigma(2)}\(\frac{r(2)}{2}\), \ldots,e^{\sigma(n-2)}\(\frac{r(n-2)}{2}\)\)
\textrm{ and } \\
u'&=&
\(e^{-\sigma(2)}\(\frac{r(2)}{2}\), \ldots,e^{-\sigma(n-2)}\(\frac{r(n-2)}{2}\)\)
\end{eqnarray*}
 are lattice points in $P_w$.  
The assertion follows. 
\end{proof}

Let $J$ be the toric ideal 
associated to the polytope $P_w$, i.e., $J$ is the kernel of the map 
$\k[X_u \mid u\in P_w \cap M]\to \k[S_{P_w}]$, where 
$X_u \mapsto (u(2),u(3),\ldots,u(n-2),1)$. 
Since $P_w$ is normal, the 
line bundle associated to $P_w$ induces a projectively normal 
embedding of the toric variety $X_{P_w}$ associated to $P_w$, with 
homogeneous coordinate ring $\k[X_u]/J$.  By Proposition~\ref{prop:sagbi}
and Lemma~\ref{lem:polytopeiso}, 
the toric variety $X_{P_w}$ is isomorphic to $\mathrm{Proj}(\inn_{\prec}(R))$.

\begin{definition}
Let $m = \prod_{t=1}^{\ell} X_{u_t}$ be a monomial in $\k[X_u]$. 
We define the  \emph{norm} of $m$ to be 
$$N(m) =
\sum_{t=1}^{\ell}\parallel u_t\parallel^2 =  \sum_{t=1}^{\ell} \sum_{i=2}^{n-2} u_t(i)^2.$$
\end{definition}

\begin{definition}
We say that a monomial $m$ is \emph{norm-minimal}, 
if for all $m'$ with $m'-m\in J$, we have $N(m')\geq N(m)$.  
\end{definition}

The following lemma characterizes norm-minimal monomials. 

\begin{lemma}\label{lem:minimal}
A monomial 
$m$ is norm-minimal if and only if 
for all $X_vX_{v'}$ dividing $m$, we have $|v(i)-v'(i)|\leq 2$
for all $2\leq i\leq n-2$. 
\end{lemma}

\begin{proof}
	To prove the Lemma, we will need the following fact:

	\begin{tabular}{cp{13cm}}
$(\star)$ & 
Let $a_i, b_i \in 2\Z$ with $a_1+\cdots+a_{\ell} = b_1+\cdots+b_{\ell}$, and $|a_i-a_j|\leq 2$ for all $i,j$. Then $\sum_{i=1}^{\ell} a_i^2 \leq \sum_{i=1}^{\ell} b_i^2$. Moreoever, equality holds if and only 
if $\{a_1, \ldots, a_{\ell}\} = \{b_1, \ldots, b_{\ell}\}$. 
\end{tabular}

Given $(\star)$, assume that $m$ is norm-minimal, but that there
exists a quadratic factor $X_vX_{v'}$ of $m$ and $2\leq i \leq n-2$ 
such that $|v(i) - v'(i)| > 2$. By Lemma~\ref{lem:normal}, there 
are $u,u'\in P_w\cap M$ with $v+v' = u+u'$ and $|u(i)-u'(i)|\leq 2$ 
for all $i$. So $X_uX_{u'}-X_vX_{v'}\in J$, and by $(\star)$, 
$N(X_uX_{u'}) < N(X_vX_{v'})$. Thus for $m'=m\frac{X_uX_{u'}}{X_vX_{v'}}$, we have $m'-m\in J$, but $N(m') < N(m)$, a contradiction. 
The converse follows immediately from $(\star)$. 

For $(\star)$, note that since $|a_i-a_j|\leq 2$ for all 
$i,j$, there exists $\alpha$ such that $a_i\in \{\alpha, \alpha+2\}$
for all $i$. After renumbering, we may assume $a_1=\cdots = a_p=\alpha$ and $a_{p+1}=\cdots = a_{\ell} = \alpha+2$. If we set $b_i = a_i+k_i$, then $\sum k_i = 0$, and we have $\sum_{i=1}^{\ell} b_i^2 - \sum_{i=1}^{\ell} a_i^2 = 
\sum_{i=1}^{\ell} k_i^2 + \sum_{i=p+1}^{\ell} 4k_i$. 
Note that when $k_i\leq -4$, then $k_i^2 + 4k_i \geq 0$, and 
if $k_i \geq 0$, then $4k_i \geq 0$, so it suffices to show that 
if $\sum k_i \geq 0$, $k_i \in 2\Z$, and $k_i =-2$ for
$p+1 \leq i\leq \ell$, then $\sum_{i=1}^{\ell} k_i^2  + \sum_{i=p+1}^{\ell}4k_i \geq 0$. This in turn follows from the fact that for $k\in 2\N$ we 
have $k^2 \geq 
 2k$, so 
for $k_i\in 2\N$ with $\sum_{i=1}^{p} k_i \geq 2 (\ell-p)$, we have
$\sum_{i=1}^{p} k_i^2   
\geq \sum_{i=1}^p 2k_i \geq 4(\ell -p)$. 
Now suppose equality holds, so 
$\sum_{i=1}^{\ell}k_i^2 + \sum_{i=p+1}^{\ell} 4k_i = 0$ where $k_i \in 2\Z$ and 
$\sum k_i = 0$. 
Note that $k_i^2 + 4k_i$ is non-negative unless
$k_i=-2$. If $R = \{ i\mid k_i =-2\}$ 
and $T=\{i \mid k_i >0\}$, then  we must have $\sum _{i\in T}
k_i \geq 2|R|$ and $\sum_{i\in T} k_i^2 + \sum_{i\in T, 
i\geq p+1} 4k_i\leq 4|R|$, but the 
only situation when this holds is when $k_i =2$ for all $i\in T$, 
no element in $T$ is larger than $p$, and $|T|=|R|$. 
The claim follows. 
\end{proof}

We now proceed to define a term order 
on the monomials in the variables $X_u$, $u \in P_w$. 

We first use the  standard lexicographic ordering $<_{\lex}$ on 
$M\cong \Z^{n-3}$  to 
order the variables $X_u, u\in P_w$.
 Let $\prec_{\grevlex}$ 
be the graded reverse lexicographic order on $k[X_u \mid u\in P_w \cap M]$ induced 
by this ordering of the variables, i.e., $m'\prec_{\grevlex} m$ if $\deg(m') <
\deg(m)$ or $\deg(m') = \deg(m)$ and for the smallest variable 
where the exponents of $m$ and $m'$ differ the exponent of $m'$ is larger than the exponent of $m$. 

We define
$m' \prec m$ iff
\begin{itemize}
\item $\deg(m') < \deg(m)$, or
\item $\deg(m') = \deg(m)$ and $N(m') < N(m)$, or
\item $\deg(m') = \deg(m)$, $N(m') = N(m)$, and
$m'\prec_{\grevlex} m$.
\end{itemize}

We shall consider two types A, B of quadratic binomial relations.

\begin{definition}(Type A)
The type A relations are relations 
$X_vX_{v'} - X_{u}X_{u'} \in J$, where
$ N(X_{v}X_{v'})> N(X_u X_{u'})$. 
\end{definition}

\begin{definition}(Type B at position $j$)
Suppose that $u ,v \in P_w \cap M$, and $3 \leq j \leq n-2$.
Suppose that 
$(u(j-1),w_{j},v(j))$ and $(v(j-1),w_{j},u(j))$ satisfy the 
triangle inequalities.  
Let $$u'= (u(2),\ldots,u(j-1),v(j),\ldots,v(n-2)),$$
$$v' = (v(2),\ldots,v(j-1),u(j),\ldots,u(n-2)).$$   
We call $X_uX_{v}- X_{u'}X_{v'}$ a type B relation at position $j$.
\end{definition}

Note that $u', v'\in P_w$ and $u+v = u'+v'$, so a type $B$ relation is 
well-defined. 
Moreover, 
$N(X_uX_v) = N(X_{u'}X_{v'})$ for any relation $X_uX_v - X_{u'}X_{v'}$ of type B.

\begin{theorem}\label{thm:GB}
The relations of type A and B form a quadratic 
Gr{\"o}bner basis for the ideal $J$. 
\end{theorem}

\begin{proof}
It suffices to show that 
\[ \In_{\prec}(J) = \< \In_{\prec}f \mid f \text{ is a type A 
or a type B relation}\>.\]
Let $m\in \In_{\prec}J$. 
Suppose $m$ is not norm-minimal. 
By  Lemma~\ref{lem:minimal}  there exists  
a quadratic factor $X_vX_{v'}$ dividing $m$ such that 
$|v(i) -v'(i)|> 2$ for some $i$. By Lemma~\ref{lem:normal} there is 
$u,u'\in P_w\cap M$ such that $X_vX_{v'}-X_uX_{u'}\in J$, and 
$|u(i)-u'(i)| \leq 2$ for all $i$. By Lemma~\ref{lem:minimal}, 
$X_uX_{u'}$ is norm-minimal, and $X_vX_{v'}$ is not, so 
$X_vX_{v'}-X_uX_{u'}$ is a type $A$ relation and 
 $m \in \<X_vX_{v'}\> \subset 
\< \In_{\prec}f \mid f \text{ is a type A relation}\>$.  

Suppose $m$ is norm-minimal. 
Let $f\in J$ be such that $\In_{\prec}f=m$. Since $J$ is a 
homogeneous binomial ideal, we may assume that $f=m-m'$, where $m'$ is a norm-minimal monomial 
of the same degree as $m$. 
Then 
$N(m') = N(m)$, but  $m' \prec_{grevlex} m$.  
Let $m = X_{u_1}X_{u_2}\cdots X_{u_\ell}$, where 
$u_1\leq_{lex} u_2 \leq_{lex} \cdots \leq_{lex} u_\ell$
and $m' = X_{v_1} X_{v_2} \cdots X_{v_\ell}$ where 
$v_1 \leq_{lex} v_2 \leq_{lex} \cdots \leq_{lex} v_\ell$.
Note that since $m-m'\in J$, we have 
\begin{equation}\label{eq:idealmembership}
 \sum _{s=1}^{\ell} u_s = \sum _{s=1}^{\ell}v_s.
\end{equation} 
Factoring out the largest common multiple of $m$ and $m'$, we 
may assume $v_1<_{lex} u_1$.  
Let $j$ be the first index 
where $v_1(j) < u_1(j)$. 
Note that if $i<j$ then $u_1(i) = v_1(i)$.

It follows from Lemma \ref{lem:minimal} and 
from \eqref{eq:idealmembership} that 
for every $2\leq k\leq n-2$ there exists an even integer 
$\alpha_k$ such that for every $X_u$ dividing $m$ or $m'$, 
$u(k) \in \{ \alpha_k, \alpha_k+2\}$.

This implies that  
$v_1(j) = \alpha_j$ and $u_1(j)=\alpha_j + 2$.
It follows from (\ref{eq:idealmembership}) 
that there exists some $t > 1$ such that $u_t(j) = v_1(j) = \alpha_j$.  
Note that, since $u_1<_{\lex} u_t$ and $u_1(j) > u_t(j)$, there is $i'<j$ such that  $u_1(i')<u_t(i')$ and $u_1(i)=u_t(i)$ for 
$i<i'$. In particular,  
$j>2$. 
Now $(u_1(j-1),w_j,u_t(j)) = (v_1(j-1), w_j, v_1(j))$ satisfy 
the triangle inequalities, 
since $v_1\in P_w$. 
If $(u_t(j-1), w_j, u_1(j))$ satisfy the 
triangle inequalities, then there exists a type B relation 
at position $j$  of the form $X_{u_1}X_{u_t}- X_{u_1'}X_{u_t'}$. 
Since $u_1' <_{\lex} 
u_1 <_{\lex} u_t$,  we have 
$X_{u_1}X_{u_t} \succ_{grevlex} X_{u_1'}X_{u_t'}$. 
Therefore,  $m\in \In _{\prec}\{ f \mid f \textrm{ is a type B relation\}}$.

Now suppose $(u_t(j-1), w_j, u_1(j))$ do not satisfy the 
triangle inequalities. 
This implies that whenever $v_s(j-1)= u_t(j-1)$, then 
$v_s(j) \neq u_1(j)$, since $(v_s(j-1), w_j, v_s(j))$ satisfy the 
triangle inequalities and so
 $v_s(j) = v_1(j)$. 
Note however that  $v_1(j-1)=
u_1(j-1) \neq u_t(j-1)$ by assumption. 
In particular, 
\[ 
\{ s \mid 1\leq s \leq \ell, v_s(j-1) = u_t(j-1)\} \subsetneq
\{ s \mid 1\leq s \leq \ell, v_s(j) = v_1(j)\}.\]

Since $\sum u_s(j-1) = \sum v_s(j-1)$, we have
\[ \#\{s \mid 1\leq s \leq \ell, v_s(j-1) = 
u_t(j-1)\} = \# \{s \mid 1 \leq s \leq \ell, 
u_s(j-1) = u_t(j-1) \} \]
and similarly \[ 
 \#\{s \mid 1\leq s \leq \ell, v_s(j) = 
v_1(j)\} = \# \{s \mid 1 \leq s \leq \ell, 
u_s(j) = v_1(j) \}. \]
 
Using $u_t(j) = v_1(j)$, we obtain
\[ 
\# \{ s \mid 1\leq s \leq \ell, u_s(j-1) = u_t(j-1)\}<
\# \{ s \mid 1\leq s \leq \ell, u_s(j) = u_t(j)\}.\] 

Pick an $s$ such that $u_s(j-1) \neq 
u_t(j-1)$ but $u_s(j) = u_t(j)$.
Since $u_1(j) \neq u_t(j)$, $s\neq1$. 
By assumption, $u_1(j-1)\neq u_t(j-1)$, hence $u_s(j-1) = u_1(j-1)$ 
since $\{u_s(j-1), u_1(j-1), u_t(j-1)\} \subset \{\alpha_{j-1}, \alpha_{j-1}+2\}$. 
So $(u_s(j-1), w_j, u_1(j)) 
= (u_1(j-1), w_j, u_1(j))$
satisfy the triangle inequalities.   On the other hand,  since $u_s(j) = u_t(j) = v_1(j)$, 
we have that $(u_1(j-1), w_j, u_s(j))  = 
(v_1(j-1), w_j, v_1(j))$ satisfy the triangle inequalities. So there exists a type B relation  
$X_{u_1}X_{u_s} - X_{u_1'}X_{u_s'}$ at position $j$. Note that $X_{u_1}X_{u_s}
\succ_{\grevlex} X_{u_1'}X_{u_s'}$, so 
$m\in \In _{\prec}\{ f \mid f \textrm{ is a type B relation\}}$.
\end{proof}

\begin{proposition}\label{prop:rad}
The initial ideal $\In_\prec(J)$ is radical.
\end{proposition}

\begin{proof}
Since the initial ideal is generated by quadratic monomials, we only 
need to check there is no perfect square $X_v^2 \in \In_\prec(J)$.  
Suppose that $X_v^2 - X_u X_w \in J$ is a non-zero 
relation for some $u$, $v$, $w$.  
Then $2v = u+w$, and so $N(X_v^2) < N(X_u X_w)$, 
since by  
Lemma~\ref{lem:minimal}
$X_v^2$ is norm-minimal, 
but $X_uX_w$ is not.
So $X_v^2$ is no  leading term of any binomial in $J$. 
\end{proof}

Note that this implies that $P_w$ admits a regular unimodular 
triangulation, see \cite[Corollary 8.9]{Sturmfels96}

\begin{proof}[Proof of Theorem~\ref{thm:main}]  
	By Kempe's theorem 
	\cite[Theorem 2.3]{HMSV}, $R_w$ is generated by 
	its lowest degree invariants. Note that for 
	 $w\in (2\Z)^n$ this also follows from Lemma~\ref{lem:normal}, 
Proposition~\ref{prop:sagbi}, and Lemma~\ref{lem:polytopeiso}.
	Recall that $I_w$ denotes the ideal of relations between these 
	generators. Since the polytopal semigroup algebra $\k[S_{P_w}]$ 
	is a SAGBI degeneration of $R_w$ by Proposition~\ref{prop:sagbi}, and since the toric ideal $J$  associated to $P_w$ has a square free quadratic initial ideal by Theorem~\ref{thm:GB} and Proposition~\ref{prop:rad}, the claim 
	follows from \cite[Corollary 2.2]{ConcaHerzogValla}. 
	The Koszul property follows for example from \cite[Proposition 3]{EisenbudReevesTotaro}.
\end{proof}

\begin{remark}\label{rem:Manon}
	In \cite[Theorem 1.10]{Manon12} Manon 
	generalizes Theorem~\ref{thm:main} to certain subpolytopes of $P_w$. For $L\geq 0$, the polytope
	$P_L$ is given 
	by the adding to the inequalities of Definition~\ref{def:P}
	the inequalities $w_1 + w_2 + u(2) \leq 2L, u(n-2)+w_{n-1}+w_n \leq 2L, u(i-1)+ w_i + u_i \leq 2L$. Note that when $L$ is large, then 
	$P_w=P_L$. For $L=1$, the polytopes $P_L$ are slices 
	of the polytopes studied in \cite{BuczynskaWisniewski}.
\end{remark}

\begin{remark}
	 It is easy to see that for $w=1^n$, $P_w$ is 
	 reflexive and that  $S_{P_w}$ is Gorenstein. 
	 This implies that the toric variety $V$ associated 
	 to $P_w$ is arithmetically Gorenstein and Fano. 
In particular,  $V$ has 
 canonical singularities. However,  
 when $w=1^6$, then $V$ does not have terminal 
 singularities. 
 Compare also \cite[Proposition 1.4]{Nill05a}.  The Gorenstein 
	 property for the toric varieties arising as degenerations 
	 of $R_w$ corresponding to arbitrary 
	 trivalent trees was studied by Manon in  
	 \cite{Manon08}. 
\end{remark}

\bibliographystyle{plain-annote}
\bibliography{/Users/milenahering/Math/mybib}
\end{document}